\documentclass[12pt]{article}
\usepackage{amsmath}
\usepackage{amsfonts}
\usepackage{amssymb}
\usepackage{graphicx}
\newtheorem{theorem}{Theorem}[section]
\newtheorem{lemma}{Lemma}[section]

\newenvironment{proof}{\smallskip\noindent{\it Proof}}{$\Box$}
\numberwithin{equation}{section}
\newcommand{\balpha}{{\boldsymbol \alpha}}
\newcommand{\bmu}{{\boldsymbol \mu}}
\newcommand{\blambda}{{\boldsymbol \lambda}}

\begin{document}

\title{Superresolution in the maximum entropy\\ approach to invert Laplace transforms} 
\author{Henryk Gzyl,\\
\noindent 
Centro de Finanzas IESA, Caracas, (Venezuela)\\
 henryk.gzyl@iesa.edu.ve}

\date{}
 \maketitle

\setlength{\textwidth}{4in}

\vskip 1 truecm
\baselineskip=1.5 \baselineskip \setlength{\textwidth}{6in}
\begin{abstract}
The method of maximum entropy has proven to be a rather powerful way to solve the inverse problem consisting of determining a probability density $f_S(s)$ on $[0,\infty)$ from the knowledge of the expected value of a few generalized moments, that is, of functions $g_i(S)$ of the variable $S.$ A version of this problem, of utmost relevance for banking, insurance, engineering and the physical sciences, corresponds to the case in which $S \geq 0$ and $g_i(s)=\exp(-\alpha_i s),$ th expected values $E[\exp-\alpha_i S)]$ are the values of the Laplace transform of $S$ the points $\alpha_i$ on the real line.\\
Since inverting the Laplace transform is an ill-posed problem, to devise numerical tecniques that are efficient is of importance for many applications, specially in cases where all we know is the value of the transform at a few points along the real axis.  A simple change of variables transforms the Laplace inversion problem into a fractional moment problem on $[0,1].$ It is remarkable that the maximum entropy procedure allows us to determine the density on $[0,1]$ with high accuracy. In this note, we examine why this might be so.
\end{abstract}
\section{Introduction and Preliminaries} 
The problem of determining the probability density of a positive random variable from the knowledge of its Laplace transform is quite an ubiquitous problem. It can be stated as
\begin{equation}\label{prob1}
\mbox{Find}\;\;\;\;f_S(s)\;\;\;\mbox{such that}\;\;\;E[e^{-S\alpha_i}] = \int_0^\infty e^{-s\alpha_i}f_S(s)ds = \mu_i\;\;\;i=1...,K.
\end{equation}
A couple of situations in which this problem appears are the following. When determining the distribution of the exit time of a diffusion from a domain $D,$ the left hand side is actually a function $\phi(x, \alpha)$ for $x \in D$ and $\alpha > 0.$ The function $\phi(x,\alpha)$ can be determined numerically by solving a boundary problem or by Montecarlo simulation. In this case one is interested in determining $f_S(x,s)$ from the knowledge of $\phi(x,\alpha)$ at a collection of values of $x\in D$ for a few values of $\alpha.$ See Gzyl and ter Horst (2013) and references therein for example. 

Another class of problems, originating in the insurance and banking industries, in which (\ref{prob1}) is highly relevant, consists of the determination of the probability density of a random variable describing a loss. In this case either from a mathematical model, or from collected data, one computes numerical the left hand side of (\ref{prob1}) for a  small number of values of the parameter $\alpha.$ This problem was addressed recently in Gzyl, Novi-Inverardi and Tagliani (2013) where a comparison between maximum entropy based methods and other techniques was carried out, and in Gomes-Gon\c{c}alves, Gzyl and Mayoral (2015) where a similar problem was addressed, but where the moments were obtained from numerical data.

In all of these references, it was observed that $8$ fractional moments sufficed to obtain very good approximations to the actual or to the empirical densities. The hint as to why that might be true comes from Gzyl and Tagliani. There we observed that in going from $4$ to $8$ moments, the entropy changed very little. Since the convergence in entropy and the convergence in the $\L_1$ norm are related, the fact that the entropy stabilizes rapidly as the number of fractional moments increases, may be responsible of the high accuracy in the reconstructions.\\
 To make these comments more precise, let us denote by $f_K$ the density reconstructed by means of the maxentropic procedure, determined by a collection of $K$ moments. Below we shall see that when the difference in entropies between successive density reconstructions $f_K$ becomes smaller, then the $f_K$ converge in $L_1$ to a density $f_\infty,$ which in our case it will be proved to coincide with $f_S.$
This argument was used by Frontini and Tagliani (1997), in which they noted that successive integer moments lead to densities having entropies that decreased and became quite close, but their argument has a gap that is filled here.

The fact that the density can be recovered from little data is called superresolution. In our version of the Laplace transform inversion, this amounts to say that the Laplace transform con be inverted from the knowledge of its values at a few points. To mention a few references devoted to the superresolution phenomenon in the context of the maximum entropy method, consider Gamboa and Gassiat (1994) and Gamboa-Gassiat (1996), Lewis (1997) and more recently de Castro et al (2015). Superresolution is also used to describe the situation in which  fine detail in a signal is obtained from a response in which much less detail is available. As examples of work along this line, consider the work by Donoho and Stark (1989), Candes, Romberg and Tao (2005) and Candes and Fernandez-Granda (2012).

We should as well mention that the use of fractional moments to invert Laplace transforms, using a more algebraic approach combined with regularization techniques has received considerable attention. Consider for example the work by Dung et al (2006) and references therein. The possibility of replacing the problem of numerically inverting the Laplace transform from a few real values of the parameter, by transforming the problem into a fractional moment problem and then applying a maxentropic approach, seems to have been first proposed in Tagliani and Vel\'asquez (2003).\\
We also mention that the Laplace inversion problem is a particular case of the problem of solving integral equations. We cite at least two attempts at solving this generic problem using two quite different techniques, but both based on the maximum entropy method. Consider Mead (1986) for an application of the standard method of maximum entropy,  and Gamboa and Gzyl (1997) for an application of the method of maximum entropy in the mean. The superresolution problem was not addressed in any of these works.

\subsection{Problem statement}
To state the problem that concerns us here, observe that the change of variables $s \rightarrow y = e^{s}$ maps (\ref{prob1}) onto
\begin{equation}\label{prob2}
\mbox{Find}\;\;\;\;f_Y(y)\;\;\;\mbox{such that}\;\;\;E[Y^{\alpha_i}] = \int_0^1 y^{\alpha_i}f_Y(y)dy = \mu_i\;\;\;i=1...,K.
\end{equation}
\noindent where of course, $Y = e^{-S},$ and once $f_Y(y)$ is obtained, $f_S(s) = e^{-s}f_Y(e^{-s})$ provides us with the desired $f_S.$ When the $\alpha_i =i$ are the non negative integers,  it is known that there exits a unique density satisfying the constraint for all $i \geq 0,$ and the conditions for this to happen are established in a variety of places. See Karlin and Shapley (1953) or Shohat and Tamarkin (1950) for example. When the moments are positive real numbers, we have the following two results established by Lin (1997), based on a basic result asserting that an analytic function on the right half complex plane is determined by its values on a sequence of points having an accumulation point there. The result that we need from Lin is

\begin{theorem}\label{Lin}
Let $F_Y$ be the distribution of a random variable $Y$ taking values in $[0,1].$ Let $\alpha_n$ be a sequence of positive and distinct numbers satisfying $lim_{n\rightarrow\infty} \alpha_n = 0$ and $\sum_{n\geq1}\alpha_n = \infty.$ Then the sequence of moments $E[Y^{\alpha_n}]$ uniquely determines $F_Y.$
\end{theorem}
So, let us suppose that we are given $\alpha_1>\alpha_2>...>\alpha_K$ are the first $K$ terms of a sequence satisfying the statement in Lin's theorem (\ref{Lin}). In the next section we shall see how the method of maximum entropy can be invoked to obtain a solution $f_K$ to the (truncated) fractional moment problem (\ref{prob1}) and in Section 3, we shall examine the convergence of the $f_K$ to $f_S.$ An interesting feature of the fractional moment problem, related to the original Laplace transform inversion problem, comes through Abelian/Tauberian theory, which relates the decay at large values of the original variable to small values of the transform parameter, and the behavior at small values of the original variable to large values of the transform parameter.

\section{The maximum entropy solution to the truncated moment problem}
The application of the maxentropic procedure to problems like (\ref{prob2}) is rather standard.
The formulation of the maximum entropy method as a variational problem seems to have been first proposed by Jaynes (1957), but it appears as well in the work by Kullback (1968). The mathematical nuances of the problem, in particular, the uses of duality were considered in Mead and Papanicolau (1984), Borwein and Lewis (1991). See Borwein and Lewis (2000) for a detailed exposition. To state our problem we follow the last cited authors and define
The entropy of a density $g(y)$ on $[0,1]$ is defined by $S:L_1([0,1],dx)\rightarrow[-\infty,+\infty)$ by
\begin{equation}\label{entr1}
S(g) = -\int_0^1 g(y)\ln g(y)dy
\end{equation}
\noindent or $-\infty$ when the integral fails to converge. The following result is established in Borwein and Lewis (2000).
\begin{lemma}\label{BL}
The entropy $S$ is a proper, upper semi-continuous, concave function of its domain
$$Dom(S) := \{ g \in L_1([0,1])| S(g) > -\infty\}$$
with weakly compact sets $\{g \in L_1([0,1]) | S(g) \geq a\},$ for all $a\in\mathbb{R}.$
\end{lemma}
We shall also make use of the related concept of Kullback divergence (or relative entropy), defined on the class of densities by 
\begin{equation}\label{kull}
\mbox{For densities}\;\; f \;\;\mbox{and}\;\;g\;\;\in L_1([0,1])\;\;\mbox{set}\;\;\;K(f,g)=\int_0^1 f(y)\ln\left(f(y)/g(y)\right)dy.
\end{equation}
A property of $K(f,g)$ relevant for what comes below is contained in
\begin{lemma}\label{propkull}
The Kullback divergence satisfies\\
1) $K(f,g) \geq 0$ and $K(f,g)=0 \Leftrightarrow f=g\;\; a.e.$\\
2) $K(f,g) \geq (1/2)\|f - g\|^2.$
\end{lemma}
The first drops out of Jensen's inequality and the second is an exercise in Kullback (1968). Nevertheless, see Borwein and Lewis (1991) for a different proof. 

It is well known that the density that maximizes (\ref{entr1}) subject to the set of constraints in (\ref{prob2}), admits the representation
\begin{equation}\label{repr1}
f_K(y) = \frac{e^{-\langle\blambda^*,y^{\balpha}\rangle}}{Z(\blambda^*)}.
\end{equation}
\noindent we added the subscript $K$ to mean that $f_K$ solves (\ref{prob2}) and satisfies the constraint given by the first $K$ momenta. Here we denote by $y^{\balpha}$ the $K-$vector with components $y^{\alpha_i}: i=1,...,K,$ and $\langle\mathbf{a},\mathbf{b}\rangle$ denotes the standard Euclidean product of the two vectors. Duality theory enters to establish that $\blambda^*$ is to be obtained minimizing the (strictly) convex function
\begin{equation}\label{dual}
\Sigma(\blambda,\bmu) = \ln Z(\blambda) + \langle\blambda,\bmu\rangle,
\end{equation}
\noindent in this case, over all $\blambda \in \mathbb{R}^K.$ Not only that, when the minimum of $\Sigma(\blambda,\bmu)$ is reached at $\blambda^*,$ we have
\begin{equation}\label{entr2}
S(f_K) = -\int_0^1 f_K(y)\ln f_K(y)dy = \Sigma(\blambda^*,\mu) = \ln Z(\blambda^*) + \langle\blambda^*,\bmu\rangle .
\end{equation}
 
\section{The superresolution phenomenon}
Let us begin with the simple 
\begin{lemma}\label{comp}
Let $M > K$ and let $f_M$ and $f_K$ be the maxentropic solution of the truncated moment problems (\ref{prob2}). Then
$$K(f_M,f_K) = S(f_K) - S(f_M) \geq 0.$$
$$\|f_M - f_K\|^2 \leq 2K(f_M,f_K).$$
\end{lemma}
The second assertion is part of Lemma (\ref{propkull}), and the first identity in the first assertion follows from (\ref{entr2}). The inequality is backed by either Lemma (\ref{propkull}), but more importantly, it follows from the fact that both $f_K$ and $f_M$ share the first $K$ moments and $S(f_K)$ has the largest entropy among such densities. 

Let the fractional moments $\{\mu(\alpha_i)| i\geq 1\}$ be determined by a sequence $\{\alpha_i| i\geq 1\}$ satisfying the conditions of Lin's Theorem (\ref{Lin}). We now state a key assumption for the rest of the section.\\
\begin{equation}\label{assump}
\mathbf{Suppose}\;\; \mbox{that the density}\;\; f\;\;\mbox{ with moments}\;\; \mu(\alpha_i)\;\;\mbox{ has finite entropy}\;\; S(f).
\end{equation}
This assumption is similar to the finiteness assumption in Corollary 3.2 to Theorem 3.1 in Borwein-Lewis (1991). We now come to the main result of this note.
\begin{theorem}\label{theo1}
Suppose that (\ref{assump}) is in force. Then, with the notations introduced above we have:\\
{\b 1} $S(f_K)$ decreases to $S(f)$ as $K\rightarrow\infty.$\\
{\bf 2} $\|f_K - f\| \rightarrow 0$ as $K\rightarrow\infty.$
\end{theorem}
\begin{proof}
An argument similar to the one used to prove Lemma (\ref{comp}), and from assumption (\ref{assump}) it readily follows that, since  $S(f_K) \geq S(f),$ then the decreasing sequence $S(f_K)$ converges. Therefore, from the second assertion in (\ref{comp}), it follows that there is a function $f_\infty$ such that $\|f_K-f_\infty\|\rightarrow 0.$ That $f_\infty$ integrates to $1$ is clear, and by taking limits along a subsequence if need be, we conclude that $f_\infty \geq 0,$ and therefore, that $f_\infty$ is a density.

Observe now that for any fixed $K$ and the corresponding  $\alpha_i, 1\leq i\leq K,$ we have
$$|\mu(\alpha_i) -\int_0^1 y^{\alpha_i}f_\infty(y)dy| = |\int_0^1 y^{\alpha_i}\left(f_M(y)-f_\infty(y)\right)dy| \leq \|f_M - f_\infty\| \rightarrow 0.$$
That is, the moments of $f_\infty$ coincide with those of $f,$ and according to Theorem (\ref{Lin}), we obtain that $f_\infty = f,$ thus concluding the proof.
\end{proof}\\
{\bf Comments} This result is similar to Theorem 3.1 in Borwein-Lewis (1991), and (\ref{assump}) is what allows us to make sure that the sequence of decreasing entropies $S(f_K)$ is actually a Cauchy sequence. This detail closes the gap in the argument in Frontini and Tagliani (1997), and provides another approach to the problem considered by Lewis (1996). 

As far as the application of Theorem (\ref{theo1}) to our numerical experiments goes, what we observed is that in going from $4$ to $8$ decreasing fractional moments, the entropy of the reconstructed densities changed very little, and when we had histograms to begin with, the fit of the maxentropic density to the histogram was quite good. This makes the application of the maxentropic procedure to the problem of reconstructing densities from Laplace transforms quite a convenient procedure compared to other numerical inversion techniques.

\end{document}